\documentclass[12pt, pdftex]{amsart}
\usepackage[margin=1in]{geometry}
\usepackage[utf8]{inputenc}
\usepackage{tikz}
\usepackage[colorlinks,anchorcolor=blue,citecolor=blue,linkcolor=blue,urlcolor=blue,bookmarksdepth=2]{hyperref}
\allowdisplaybreaks

\usepackage{subcaption}
\usepackage{graphics}
\usepackage[all]{xy}
\usepackage{comment}
\usepackage{amsmath}
\usepackage{amssymb}
\usepackage{amsthm}
\usepackage{here}
\usepackage{amscd} 
\usepackage{tikz-cd}
\usepackage{mathrsfs}
\usepackage{mathtools}
\usepackage{mathabx}
\usepackage{scalefnt}
\usepackage{url}
\usepackage{breqn}
\theoremstyle{plain}
\newtheorem{thm}{Theorem}[section]
\newtheorem{lem}[thm]{Lemma}
\newtheorem{cor}[thm]{Corollary}
\newtheorem{prop}[thm]{Proposition}
\newtheorem{conj}[thm]{Conjecture}
\theoremstyle{definition}

\newtheorem{rem}[thm]{Remark}

\newtheorem{defn}[thm]{Definition}
\newtheorem{prob}[thm]{Problem}
\newtheorem{que}[thm]{Question}

\numberwithin{equation}{section}
\def\A{{\mathbb A}}

\def\Q{{\mathbb Q}}

\def\C{{\mathbb C}}
\def\P{{\mathbb P}}
\def\B{{\mathbb B}}

\def\Aut{\mathop{\mathrm{Aut}}\nolimits}

\def\SL{\mathop{\mathrm{SL}}\nolimits}

\def\ord{\mathrm{ord}}

\def\L{\mathscr{L}}

\def\M{\mathcal{M}}

\def\OO{\mathscr{O}}

\def\A{\mathcal{A}}

\def\a{\alpha}
\def\b{\beta}

\def\U{\mathrm{U}}

\allowdisplaybreaks[4]

\newcommand{\defeq}{\vcentcolon=}

\begin{document}

\title[Remarks on two problems by Hassett]
{Remarks on two problems by Hassett}
\dedicatory{Dedicated to Yuri Tschinkel on the occasion of his 60th birthday}

\author[Klaus Hulek and Yota Maeda]{Klaus Hulek$^{1}$ \and Yota Maeda$^{2, 3}$}
\email{hulek@math.uni-hannover.de,\quad y.maeda.math@gmail.com}

\date{\today}

\maketitle
\vspace{-1em}
\begin{center}
  \begin{minipage}{0.9\textwidth}
    \centering
    {\small
    $^{1}$ Institut f\"ur Algebraische Geometrie, Leibniz University Hannover, Germany.\\
    $^{2}$ Fachbereich Mathematik, Technische Universit\"at Darmstadt, Germany.\\
    $^{3}$ Mathematical Institute, Tohoku University, Japan.
    }
  \end{minipage}
\end{center}

\vspace{1em}

\date{\today}

\maketitle
\begin{abstract}
One of the ultimate goals of the Hassett-Keel program is the determination of the log canonical models of the moduli spaces of pointed rational curves $\overline{\M}_{0,n}$.
In this paper, we study log canonical models of $\overline{\M}_{0,5}$ with \textit{asymmetric} boundary divisors.
Our results generalize previous work by Alexeev-Swinarski, Fedorchuk-Smyth, Kiem-Moon and Simpson for the  first non-trivial case, namely $n=5$.
We prove that all moduli spaces of weighted pointed rational curves $\overline{\M}_{0,\A}$ arise as log canonical models of $\overline{\M}_{0,5}$ for suitable choices of boundary coefficients, thereby also recovering a theorem of Fedorchuk and Moon. In addition, we relate these moduli spaces to Deligne-Mostow ball quotients.
We further study log canonical models of the moduli spaces $\overline{\M}_{0,n\cdot (1/k)}$ with symmetric weight, which differ from $\overline{\M}_{0,n}$.
The case $n=5$ can be viewed as an explicit guiding example  in a very general program 
and the paper can thus also serve as an expository introduction.  
\end{abstract}

\section{Introduction}
\label{sec:introduction}
The moduli spaces $\overline{\M}_g$ of stable curves of genus $g$ were constructed by Deligne and Mumford \cite{DM69} and have been central objects of algebraic geometry ever since. The aim of the Hassett-Keel program 
is to run the log minimal model program for $\overline{\M}_g$ with its natural boundary consisting of the singular stable curves. 
Due to the work of Hassett \cite{Has05}, Hassett and Hyeon \cite{HH09,HH13} and Hyeon and Lee \cite{HL10}, this is now fully understood for $g \leq 3$. 
The case $g = 4$ was initially addressed by Casalaina-Martin, Jansen and Laza \cite{CMJL12, CMJL14}, 
and Fedorchuk \cite{Fed12}.
This was continued by Liu and Zhao \cite{LZ24}, with Ascher, DeVleming, Liu, and Wang \cite{ADLW25} completing the analysis of the remaining cases using techniques from the theory of $K$-stability.

In a parallel direction, Hassett proposed the analogous problem of determining the log canonical models of the Deligne-Mumford compactification of the moduli space of stable $n$-pointed rational curves $\overline{\M}_{0,n}$ 
with combinatorial boundary divisors.
He also provided several explicit examples \cite[Section 7, Remark 8.5]{Has03}.
In detail, let $\M_{0.n}$ be the open subset where the underlying curve is smooth and let $D\defeq \overline{\M}_{0,n} \setminus \M_{0.n}$
be the union of the boundary divisors, which is decomposed into irreducible components $D_{I,I^c}$ where $I\subset \mathbb{N}_n\defeq\{1,\cdots,n\}$ with $2\leq |I| \leq \lfloor n/2\rfloor$ and $I^c$ is its complement.
The geometric meaning of the divisor $D_{I,I^c}$ is that a generic point corresponds to stable nodal curves consisting of two irreducible components, with the marked points indexed by $I$ on one component, and those indexed by $I^c$ on the other.
Since the boundary divisors $D_{I,I^c}$ form a simple normal crossing divisor, the pair $(\overline{\M}_{0,n},\sum_I d_{I,I^c} D_{I,I^c})$ is log canonical for any choice of coefficients $0 \leq d_{I,I^c} \leq 1$.
The central problem is then to determine the log canonical model
\begin{align}
    \label{eq:lc model general}
    \overline{\M}_{0,n}\left(\{d_{I,I^c}\}_I\right)\defeq \mathrm{Proj} \bigoplus_{s\geq 0} H^0\left(\overline{\M}_{0,n}, s\left(K_{\overline{\M}_{0,n}} + \sum_I d_{I,I^c} D_{I,I^c} \right)\right).
\end{align}

\begin{prob}[{\cite[Problem 7.1]{Has03}}]
\label{prob:Hassett}
    Determine $\overline{\M}_{0,n}\left(\{d_{I,I^c}\}_I\right)$ for the coefficients $0 \leq d_{I,I^c} \leq 1$. 
\end{prob}
To recall the progress on this problem so far, we briefly review the relevant moduli spaces; see \cite[Section 2]{Has03} in detail.
For an $n$-tuple $\A=(a_i)_i\in\Q^n$ with $0 < a_i\le 1$ and $\sum_i a_i > 2$, which we call \emph{weight} in this paper, Hassett constructed a coarse moduli space $\overline{\M}_{0,\A}$ of weighted pointed rational curves with respect to the weight data $\A$ \cite[Theorem 2.1]{Has03}.
These moduli spaces $\overline{\M}_{0,\A}$ generalize the spaces $\overline{\M}_{0,n}$, which can also be interpreted as the Hassett space for all weights being equal to $1$. 
All other Hassett spaces 
 $\overline{\M}_{0,\A}$ are realized as combinatorial contractions of the boundary divisors $D_{I,I^c}$; see Section \ref{sec_VGIT}. A well-known approach to Problem \ref{prob:Hassett} is to analyse the correspondence between the variation of coefficients $d_{I,I^c}$ in the description of (\ref{eq:lc model general}) and the variation of weights $\A$ in the Hassett spaces.

Determining the log canonical model $\overline{\M}_{0,n}\left(\{d_{I,I^c}\}_I\right)$ for arbitrary coefficients $\{d_{I,I^c}\}_I$ is a challenging problem. However, significant progress has been made in the case where all coefficients $d_{I,I^c}$ are equal to a single parameter $\alpha$; we refer to this as the \textit{symmetric} case.  
In this setting, Problem \ref{prob:Hassett} has been extensively studied by 
Simpson \cite{Sim08} (assuming Fulton's conjecture), Alexeev-Swinarski \cite{AS12}, Kiem-Moon \cite{KM11} and Fedorchuk-Smyth \cite{FS11}, identifying the log canonical models 
as the moduli spaces of weighted pointed rational curves with symmetric weight $\overline{\M}_{0, n\cdot(1/k)}$, where 
\[n\cdot (1/k) \defeq \left(\frac{1}{k},\cdots, \frac{1}{k}\right)\in(\Q\cap [0,1])^n.\]

\begin{thm}[{\cite{AS12, FS11, KM11, Sim08}}]
\label{thm:lcmodel_previous}
Let $\a\in\Q$ with $2/(n-1) < \a \leq 1$.
Then, the log canonical model 
\[\overline{\M}_{0,n}(\a) \defeq \mathrm{Proj} \bigoplus_{s\geq 0} H^0\left(\overline{\M}_{0,n}, s\left(K_{\overline{\M}_{0,n}} + \a D \right)\right)\]
   satisfies the following 
\begin{enumerate}
    \item If $2/(k+2)< \a \leq 2/(k+1)$ for $1\leq k \leq \lfloor (n-1)/2\rfloor$, then $\overline{\M}_{0,n}(\a)\cong \overline{\M}_{0, n\cdot (1/k)}$.
    \item If $2/(n-1) < \a \leq 2/(\lfloor n/2 \rfloor+1)$, then $\overline{\M}_{0,n}(\a)\cong (\P^1)^n//_{\OO(1,\cdots, 1)}\SL_2(\C)$.
\end{enumerate}
\end{thm}

Here we study the situation for a concrete $n$ in detail. 
The cases $n \leq 3$ are trivial, as the moduli spaces are 
$0$-dimensional.   
Next is the case  $n=4$.
But in this case all varieties $\overline{\M}_{0,\A}$ are isomorphic as coarse moduli spaces, as we will explain in Remark \ref{rem:chamber_decomposition}.
This is the reason why we treat the case $n=5$ in this paper, as this is the first nontrivial case for Problem \ref{prob:Hassett} considered here. 
It turns out that the geometry is not so very complicated in this case, all spaces being closely related to del Pezzo surfaces. But we still think it worthwhile to describe this case explicitly, as it can also serve as an expository introduction and  a guiding example for future investigations.

When $n = 5$, and if $\A$ is symmetric, that is each component has the same value, then $\overline{\M}_{0,\A} \cong \overline{\M}_{0,5}$ as coarse moduli spaces; see Section \ref{sec_VGIT} for more details.
This observation implies the following corollary.
\begin{cor}[The case $n=5$ in Theorem \ref{thm:lcmodel_previous}]
\label{cor:n=5_previous}
For the case of $n=5$,  the log canonical model $\overline{\M}_{0,5}(\a)$ can be determined as follows:
if $1/2< \a \leq 1$, then $\overline{\M}_{0,5}(\a)\cong\overline{\M}_{0,5}\cong (\P^1)^5//_{\OO(1,\cdots,1)}\SL_2(\C)$.
\end{cor}

As we shall see,  there are 76 distinct Hassett spaces $\overline{\M}_{0,\A}$ for possible weight data $\A$ arising as coarse moduli spaces (Proposition \ref{prop: n=5 Hassett spaces}). Therefore, it is natural to ask whether they can be interpreted  
as log canonical models corresponding to boundary divisors with \textit{asymmetric} weights.

Before we enter this discussion, we recall that Hassett \cite[Section 5]{Has03} defined the {\em weight domain} of admissible weights 
\[
\mathcal D_{g,n} \defeq \{(a_1, \ldots , a_n) \in \mathbb R^n \mid 0< a_i \leq 1 \textrm{ and } \sum_i a_i > 2 -2g \}.
\]  
 He further introduced two chamber decomposition of this domain. 
These are defined by removing certain walls and then taking the connected components. In the case of the 
{\em {coarse}} chamber decomposition one removes the walls 
\[
{\mathcal W}_c \defeq \{ \sum_{j \in S} a_j=1 \mid S \subset \{1, \ldots, n\},\ 3 \leq  |S|  \leq n \}
\footnote{We recall that there was a minor error in the original definition of the walls, which was pointed out by K.~Ascher, see also the paper by Alexeev-Guy \cite[Remark 2.3]{AG08} (although this does not affect our setting $g=0$).}
\]
whereas one obtains the {\em{fine}} chamber decomposition by removing the walls 
\[
{\mathcal W}_f \defeq \{ \sum_{j \in S} a_j=1 \mid S \subset \{1, \ldots, n\},\ 2 \leq  |S|  \leq n \}. 
\]
The chambers obtained by removing these walls describe the variation of the moduli problems introduced by Hassett. Here, coarse chamber decomposition defines the coarsest decomposition 
such that $\overline{\M}_{g,\A}$ is constant on each chamber, whereas the fine chamber decomposition is the coarsest decomposition such that the universal family 
$\mathcal C_{g,\mathcal A}$ is constant on each chamber, see \cite[Proposition 5.1]{Has03}.
Hassett  proposed the problem of counting the number of chambers of $\mathcal{D}_{g,n}$: 
\begin{prob}[{\cite[Problem 5.2]{Has03}}]
\label{prob:Hassett counting of chamber decomposition}
    Determine the number of nonempty chambers in $\mathcal{D}_{g,n}$. 
\end{prob}

This is a highly non-trivial combinatorial problem. The enumeration of fine chambers was studied in some depth in \cite{ADGH20}. There, the authors provide explicit numbers for small $n$ and determine the asymptotic behaviour for large $n$. For $g=0$, there are 1087 fine chambers for $n=5$ and 105123 for $n=6$.  As we shall see in Proposition \ref{prop: n=5 Hassett spaces} the number of coarse chambers is 
considerably smaller, namely 76 for $n=5$.
From the perspective of computational complexity \cite[Theorem 1.4]{ADGH20}, it is believed to be difficult to obtain a concise formula that gives these numbers explicitly for general $n$. 
The same is expected for the enumeration of coarse chambers. 

The symmetric group $S_n$ acts on both, the domain $\mathcal D_{g,n}$ and the walls, and hence also on the set of chambers (fine and coarse). 
The number of fine chambers modulo this action equals 36 and 448 for $g=0,n=5,6$, see \cite[Figure 2]{ADGH20}. For $n=5$ the corresponding number of coarse chambers modulo $S_5$ is 6, see Proposition \ref{prop: n=5 Hassett spaces}. 

In this paper we shall only be concerned with the coarse chamber decomposition and give a very explicit answer for our test case $g=0, n=5$.
For this let
\begin{equation}
\label{equ:defI}
\mathbb{I} \defeq \{I\subset \mathbb{N}_5\mid |I| = 2\},\quad \mathbb{S} \defeq \{S\subset \mathbb{N}_5\mid |S| = 3\}.
\end{equation}
Note that $I\in \mathbb{I}$ if and only if $I^c\in\mathbb{S}$.
For $\A = (a_i)_{i=1}^5\in\mathcal{D}_{0,5}$ and $S\in\mathbb{S}$, we introduce the sum indexed by $S$ as $\A(S)\defeq \sum_{i\in S} a_i$.
We also define 
\[D(\A)\defeq \{I\in\mathbb{I}\mid \A(I^c) \le 1\}\]
and $d(\A) \defeq |D(\A)|$.

In our case we can enumerate all coarse chambers explicitly and relate the moduli spaces to del Pezzo surfaces as follows:
\begin{prop}[{Propositions \ref{prop: n=5 Hassett spaces}, \ref{prop:contraction}}]
\label{mainprop:chamber}
    The weight domain $\mathcal{D}_{0,5}$ decomposes into $76$ non-empty chambers. The set of associated moduli spaces, considered as abstract varieties, decomposes into $6$ types.
    For a given $\A\in\mathcal{D}_{0,5}$, the moduli space is a del Pezzo surface of degree $5+d(\A)$,
     where $d(\A)$ ranges from $0$ to $4$ and both del Pezzo surfaces of degree $8$ occur. 
\end{prop}

\begin{rem}
\label{rem:n=6}
    The method presented in this paper extends in principle to higher dimensions with $n \geq 6$, but the increasing number of combinations leads to much more intricate calculations.
    For example, in the case $n=6$, computer-based computation shows that there are $36368$ nonempty coarse chambers in $\mathcal{D}_{0,6}$ and $17$ chambers modulo the $S_6$-action.
\end{rem}

Corollary \ref{cor:n=5_previous} shows that when all boundary divisor coefficients are equal to a single parameter $\alpha = d_{I,I^c}$, the log canonical model is isomorphic to the unique coarse moduli space $\overline{\M}_{0,5}$ when $1/2 < \alpha \le 1$.
In this paper, we further prove that all 75 other Hassett 
moduli spaces are also realized as log canonical models of $\overline{\M}_{0,5}(\{d_{I,I^c}\}_{I\in\mathbb{I}})$ for suitably chosen \emph{asymmetric} weights $\{d_{I,I^c}\}_{I \in \mathbb{I}}$, thus refining the description provided by Corollary \ref{cor:n=5_previous}.

\begin{thm}
\label{thm:n=5}
   Let $\A\in \mathcal{D}_{0,5}$.
For any coefficients $\{d_{I,I^c}\}_{I\in\mathbb{I}}$, define $\alpha_I \defeq d_{I,I^c}$ for $I\in D(\A)$.
If $d_{I,I^c}$ are all equal to a constant $\beta>1/2$ for $I\not\in D(\A)$ and if $\alpha_I - 3\beta + 1 \geq 0$ for all $I\in\mathbb{I}$, then the log canonical model  $\overline{\M}_{0,5}\left(\{d_{I,I^c}\}_{I\in\mathbb{I}}\right)$ is isomorphic to $\overline{\M}_{0,\mathcal{A}}$, which is a del Pezzo surface of degree $5+d(\A).$
\end{thm}
Here we briefly outline the strategy of the proof of Theorem \ref{thm:n=5}, also explaining how the conditions on $\alpha_I$ and $\beta$ arise.
The two conditions in the statement come from the two main steps of the argument. First, we compare the divisor
\[
K_{\overline{\M}_{0,5}}+\sum_{I\in \mathbb{I}} d_{I,I^c}D_{I,I^c}
\]
with its pushforward to $\overline{\M}_{0,\A}$ under the reduction morphism $\rho_{\A} \colon \overline{\M}_{0,5}\to \overline{\M}_{0,\A}$. Since the exceptional boundary divisors appear with coefficient $3$ in the pullback of the non-contracted boundary part, the condition $\alpha_I-3\beta+1\ge 0$ is precisely what ensures that the difference between the original divisor and the pullback from $\overline{\M}_{0,\A}$ is effective and $\rho_{\A}$-exceptional. Second, the divisor obtained by pushing forward to  $\overline{\M}_{0,\A}$ must be nef in order to apply the basepoint-free theorem. For this, we check non-negativity on the relevant curves using Fulton's conjecture for $n=5$, and the assumption $\beta>1/2$ is exactly the condition guaranteeing this positivity. Thus the theorem is proved by writing the divisor as a pullback from $\overline{\M}_{0,\A}$ plus an effective exceptional part and then showing that the divisor downstairs is semiample.

Theorem \ref{thm:n=5} gives a partial
answer to Problem \ref{prob:Hassett} for asymmetric weights.
Indeed, the cases treated in Theorem \ref{thm:n=5} deal with all possible Hassett spaces $\overline{\M}_{0,\A}$ derived from the chamber decomposition given in  \cite[Proposition 5.1]{Has03}.
In particular, this generalizes Corollary \ref{cor:n=5_previous} to asymmetric weights. We shall give a proof of this theorem in Section \ref{section: proof of the main theorem}.

We shall further explore the relationship with Deligne-Mostow ball quotients, establishing the following proposition
\begin{prop}[Definition \ref{defn:modular period}, Proposition \ref{prop:relation to DM}]
For any weight $\A\in\mathcal{D}_{0,5}$ the moduli space $\overline{\M}_{0,\A}$ has a Deligne-Mostow modular interpretation. 
\end{prop}
In general, the relationship with ball quotients has been discussed less systematically in the literature. In Questions \ref{que:ball1} and \ref{que:ball2} we pose two problems concerning this topic.

Theorem \ref{thm:n=5} shows that all Hassett spaces appear as log canonical models of $\overline{\M}_{0,5}$ with appropriate coefficients.
Hence, we recover the following theorem, which was originally proved by Fedorchuk \cite{Fed11} and Moon \cite{Moo13,Moo15a}.
\begin{cor}
    For any $\A\in\mathcal{D}_{0,5}$, there exists a set of coefficients $\{d_{I,I^c}\}_{I\in\mathbb{I}}$ such that the log canonical model $\overline{\M}_{0,5}(\{d_{I,I^c}\}_{I\in\mathbb{I}})$ is isomorphic to $\overline{\M}_{0,\A}$.
\end{cor}

To the best of the authors' knowledge, the only other works concerning asymmetric weights in this context are \cite{Fed11,Moo13,Moo15a}, whose settings are entirely distinct and mutually exclusive from ours.
Fedorchuk \cite[Theorem 4]{Fed11} proved that for a given weight $\A$, the log canonical model $\overline{\M}_{0,n}(\{d_{I,I^c}\}_{I\in\mathbb{I}})$ coincides with $\overline{\M}_{0,\A}$ where for $I=\{i_1,i_2\}$,  $d_{I,I^c} = 1$ when $a_{i_1} + a_{i_2}\le 1$ and $d_{I,I^c} = a_{i_1} + a_{i_2}$ otherwise.
Moon's works \cite[Theorem 1.4]{Moo13} and \cite[Theorem 1.1]{Moo15a} involve psi classes and use a completely different set-up.

 Our paper deals with the two-dimensional Hassett–Keel program. In this case, $\overline{\M}_{0,5}$ is a del Pezzo surface, and every birational contraction is also a Hassett moduli space. 
One notable aspect for $n\ge 6$ is that there exist birational contractions of $\overline{\M}_{0,n}$ whose targets are not Hassett moduli spaces.
They include the symmetric GIT quotients $(\P^1)^n/\!/\SL_2$ (e.g. the Segre cubic for $n=6$), certain Veronese quotients \cite{GJMS13},  
extremal-assignment contractions \cite{MSvX18}, and toric models attached to hypergraph associahedra \cite{BGA24}. 
According to Theorem \ref{thm:lcmodel_previous}, we can, in our case, obtain the symmetric GIT quotient as a log canonical model. 
Also, it is known that all Hassett spaces are Veronese quotients (\cite[Corollary 7.2]{GJM13}).
By contrast, to the best of the authors' knowledge, although there are studies concerning the log minimal model program for certain specific values of $n$ \cite{Moo15b,Moo17}, there is currently no result for realizing Veronese quotients, extremal-assignment models, or toric hypergraph associahedra as log canonical models if they are not isomorphic to Hassett moduli spaces; see \cite[Remark 3.3 (5)]{GJMS13}.

Finally, we shall briefly touch on the log minimal model program for (smooth) del Pezzo surfaces.
   Since a del Pezzo surface $X$ is Fano, it is natural to consider a pair $(X,\Delta)$ for $\Delta \in |-K_X|$, which forces the pair to be log Fano if it is log canonical.
   Cheltsov \cite[Theorem 1.7]{Che08} computed the log canonical threshold for such a pair.
   In that range of the coefficients of $\Delta$, the log canonical model coincides with its anticanonical model, which is classically known \cite[Subsection 8.3]{Dol12}.
   We are,  however, not aware of an intrinsic relationship between these results and the results of this paper.

\subsection*{Organization of the paper.}
In Section \ref{sec_VGIT}, we review Hassett's chamber decomposition for $\mathcal{D}_{0,5}$ and classify the resulting moduli spaces $\M_{0,\A}$ into six types, showing that they are del Pezzo surfaces of 
degree $\geq 5$ and that all del Pezzo surfaces in this degree range occur. 
We also discuss their relation to GIT quotients and Deligne-Mostow ball quotients. In Section \ref{sec:Fulton's conjecture}, we recall the case $n=5$ of Fulton's conjecture, which is the key input for the nefness arguments used later. In Section \ref{section: proof of the main theorem}, we prove Theorem \ref{thm:n=5} by comparing canonical divisors under the reduction morphism $\rho_{\A} \colon \overline{\M}_{0,5}\to \overline{\M}_{0,\A}$ and applying the basepoint-free theorem. Finally, in Section \ref{sec:Log canonical models of the moduli spaces of weighted pointed rational curves with symmetric weights}, we discuss log canonical models of Hassett spaces with symmetric weights in higher dimensional cases.

\subsection*{Acknowledgements}
The authors would like to thank Brendan Hassett and Zhiwei Zheng for helpful discussions.
The second author is also grateful to Masafumi Hattori for a discussion regarding Remark \ref{rem:VGIT} and Leibniz University Hannover and the University of Nottingham for their hospitality.
The first author was also partially supported by DFG grant Hu 337/7-2. 
This work was partially supported by the Alexander von Humboldt Foundation through a Humboldt Research Fellowship granted to the second author.
This project started when both authors were Leibniz Research Fellows in Oberwolfach. We thank the Mathematische Forschungsinstitut Oberwolfach (MFO) for hospitality and for providing excellent working conditions.

\section*{Notation and convention}
We denote by $\overline{\M}_{0,\A}$ the coarse moduli space of weighted pointed curves of genus 0 as introduced in \cite{Has03}.
We call $\A$ \emph{symmetric} if all components of $\A$ have the same value.
In this case we write $n\cdot a$ as an $n$-tuple $(a,\cdots,a)$ for simplicity.
Throughout the paper, we will work over $\C$.
 
\section*{Del Pezzo surfaces}
Here we summarize basic facts concerning del Pezzo surfaces.
A del Pezzo surface is a smooth surface whose canonical bundle is anti-ample. We refer to \cite{Man86} for a classical treatment. The degree $d$ of a del Pezzo surface is the self-intersection number of the (anti)canonical bundle and is bounded by $1 \leq d \leq 9$. 
Let $S_d$ denote a del Pezzo surface of degree $d$.
If $d \neq 8$, there is a unique deformation type of $S_d$. If $d=9$, it is $\P^2$ and for $1 \leq d \leq 7$, it is 
 $\P^2$ blown up in $9-d$ points in general position. If $d=8$, $S_8$ is either $\P^2$ blown up in 1 point, which is the Hirzebruch surface $\mathbb{F}_1$, or $\P^1 \times \P^1$. For the degrees $1 \leq d \leq 4$, the moduli space of $S_d$ has dimension $10-2d$, otherwise, del Pezzo surfaces are uniquely determined by their degree. 

An interesting feature of del Pezzo surfaces is the curves of negative self-intersection they contain. These are necessarily $(-1)$-curves, and there are finitely many of these. The number $\#(d)$ of $(-1)$-curves on $S_d$ is given by the following table where 8a corresponds to $\P^1 \times \P^1$ and 8b to $\mathbb{F}_1$: 
\begin{table}[h]
  \centering
  \caption{Number of $(-1)$-curves on del Pezzo surfaces} 
    \renewcommand{\arraystretch}{1.3}
  \setlength{\tabcolsep}{8pt}
\begin{tabular}{|c|c|c|c|c|c|c|c|c|c|c|}
   \hline
   $d$ & 9 & 8a & 8b  & 7 & 6 & 5 & 4 & 3 & 2 & 1  \\ \hline
    $\#(d)$ & 0  & 0 & 1 & 3 & 6 & 10 & 16 & 27 & 56 & 240\\ 
    \hline
 \end{tabular}
\end{table}

References for these well known statements can be found in  \cite[Theorem 24.3]{Man86} and \cite[Corollary 25.5.4]{Man86}.
The configuration of $(-1)$-curves on del Pezzo surfaces is closely related to the ADE-root diagrams; see \cite[Ch. IV. \S 25]{Man86}.
If $f: S_d \to S_{d'}$ is a birational morphism between del Pezzo surfaces, then $d' \geq d$ and $f$ is a sequence of $d'- d$ contractions of $(-1)$-curves.

\section{Variation of the moduli problems for \texorpdfstring{$\overline{\M}_{0,\A}$}{M_{0,\A}}}\label{sec_VGIT}
We begin by recalling the foundational result of Hassett concerning the chamber structure resulting from the  variation of the moduli problem for the spaces $\overline{\M}_{0,\A}$. 
Throughout this paper, we restrict ourselves to the case $g=0$.
As already mentioned in the introduction, according to \cite[Proposition 5.1]{Has03}, the weight domain $\mathcal{D}_{0,n}$ is divided into finitely many connected components, called {\em (coarse) chambers}, 
by removing the walls
\[
\mathcal W_c= \left\{\sum_{j\in S} a_j =1\ \middle|\ S\subset \mathbb{N}_n,\ 3 \leq |S| \leq n \right\}.
\]
 We further recall, as we have already mentioned in the introduction, the following special case of a theorem of Hassett:
\begin{thm}[{\cite[Section 5]{Has03}}] 
\label{thm:coarse chamber}
The coarse chamber decomposition defines the coarsest decomposition of $\mathcal D_{0,n}$ such that $\overline{\M}_{0,\A}$ is constant on each chamber.
\end{thm}
\begin{rem}
\label{rem:chamber_decomposition}
To start with, we shall briefly discuss the case $n=4$.
A weight $\A=(a_1,a_2,a_3,a_4) \in \mathcal{D}_{0,4}$ must satisfy
\begin{enumerate}
    \item $0< a_i\leq 1$, 
    \item $a_1+a_2+a_3+a_4 > 2$.
\end{enumerate}
Any such quadruple $(a_1,a_2,a_3,a_4)$ necessarily satisfies $a_i+a_j+a_k > 1$ for all triplets $\{i,j,k\}\subset \mathbb{N}_4$.
Hence, all combinations $\A$ lie in the same coarse chamber, which shows that $\overline{\M}_{0,\A}$ does not depend on $\A$ as a coarse moduli space.
Therefore, there is no distinction between the weights in classifying $\overline{\M}_{0,\A}$.
\end{rem}
We now use Theorem \ref{thm:coarse chamber} to analyse the moduli spaces of weighted pointed rational curves for the first nontrivial case $n=5$.
First, we count the number of $S\in\mathbb{S}$ which define a non-empty chamber.
The following is a purely combinatorial problem.
\begin{lem}
\label{lem:combinatorial computation of I}
Let $0<a_1\le a_2\le a_3\le a_4\le a_5\le1$ and suppose
$\sum_{i=1}^5 a_i > 2$.
Then there are at most four $3$-element subsets $S\in\mathbb{S}$ for which
$\sum_{i\in S}a_i<1$.
\end{lem}

\begin{proof}
Set $\sigma\coloneqq\sum_{i=1}^5 a_i$ and define
\[
m\defeq a_1+a_2+a_3,\qquad M\defeq a_1+a_2+a_5.
\]

\textbf{Case 1: $m\ge 1$.}
Since $\{1,2,3\}$ has the smallest $3$-sum among all triples by the monotonicity
$a_1\le \cdots\le a_5$, every $3$-sum is $\ge1$, so there are no triples with sum $<1$.

\textbf{Case 2: $m<1$ and $M\ge 1$.}
Any triple having the form $\{i,j,5\}$ has a sum at least $a_i + a_j + a_5 \ge M\ge1$.
The only possible triples whose sum is less than 1 are those avoiding $5$, of which there are
$\binom{4}{3}=4$; thus there are at most four $3$-element subsets $S$.

\textbf{Case 3: $m<1$ and $M<1$.}
From $\sigma>2$, we obtain
\[
a_4+a_5=\sigma-m>1,\qquad a_3+a_4=\sigma-M>1.
\]
Consequently, any triple that contains $\{4,5\}$ or $\{3,4\}$ has sum $>1$.
In particular, for any $S\in\mathbb{S}$, which coincides with one of 
\[
\{1,4,5\},\ \{2,4,5\},\ \{3,4,5\},\ \{1,3,4\},\ \{2,3,4\},
\]
we have $\sum_{i\in S} a_i >1$.
Now, the condition $m<1$ and $M<1$ imply that $\sum_{i\in S}a_i <1$ for 
\[S=\{1,2,3\},\ \{1,2,5\},\ \{1,2,4\}.
\]
The only remaining triples are $\{1,3,5\}$ and $\{2,3,5\}$, and we deduce that $a_2 + a_3 + a_5 \ge a_1 + a_3 + a_5 > a_3+a_5\ge a_3+a_4>1$. Hence, in this case, there are exactly three triplets whose sum is less than $1$.

Combining the three cases, the number of $S\in\mathbb{S}$ with $\sum_{i\in S} a_i <1$ is always at most four.
\end{proof}

Before discussing the chamber structure of $\mathcal{D}_{0,5}$ in detail, let us recall the basic surface-theoretic picture underlying our arguments. The moduli space $\overline{\M}_{0,5}$ is isomorphic to the blowup of $\mathbb{P}^2$ at four general points, hence it is the del Pezzo surface of degree $5$. Moreover, its ten boundary divisors $D_{I,I^c}$ are exactly the ten $(-1)$-curves on this surface. Therefore, whenever one crosses a wall in the chamber decomposition so that $\A(I^c)\le 1$, the corresponding reduction morphism contracts the $(-1)$-curve $D_{I,I^c}$. In this way, the birational geometry of the spaces $\overline{\M}_{0,\A}$ can be described explicitly in terms of successive blowdowns of boundary divisors as follows.
\begin{defn}
\label{defn:classification for n=5}
    Let $\A\in \mathcal{D}_{0,5}$ be a weight and $D(\A)=\{I_1,\cdots,I_{d(\A)}\}$.
    We define the types (A) - (F) of $\A$ as follows:
        \begin{enumerate}
        \item[($A$)] $d(\A) = 0$.
        \item[($B$)] $d(\A) = 1$.
        \item[($C$)] $d(\A) = 2$.
        \item[($D$)] $d(\A) = 3$ and $|I_1\cup I_2\cup I_3| = 4$.
        \item[($E$)] $d(\A) = 3$ and $|I_1\cup I_2\cup I_3| = 3$.
        \item[($F$)] $d(\A) = 4$.
    \end{enumerate}
\end{defn}
In what follows, we show that these six types are the representatives of chambers by the $S_5$-action and enumerate them:

\begin{prop}
\label{prop: n=5 Hassett spaces}
The domain $\mathcal D_{0,5}$ has exactly $76$ chambers, of which
\begin{enumerate}
    \item $1$ is of type (A),
    \item $10$ are of type (B),
    \item $30$ are of type (C),
    \item $20$ are of type (D),
    \item $10$ are of type (E), and
    \item $5$  are of type (F).
\end{enumerate}
These correspond to the coarse chambers modulo the $S_5$-action.
\end{prop}
\begin{proof}
By the above discussion, the coarse chamber corresponds to the connected components of 
\begin{align}
\label{eq:coarse chamber decomposition}
   \mathcal{D}_{0,5}\setminus \bigcup_{S\subset \mathbb{N}_5,\ 3 \leq |S| \leq 5}\left\{\sum_{j\in S} a_j =1\right\}.
\end{align}
A priori, this leaves us the cases $|S| = 3,4,5$. The cases $|S|=4,5$, however are not possible. 
    Indeed, if $|S|=5$, then this contradicts the condition $\sum_{i=1}^5 a_i > 2$ as $S$ coincides with $\mathbb{N}_5$. 
    Further, if $|S|=4$, then there exists an index $k$ so that $\sum_{j\in \mathbb{N}_5\setminus \{k\}} a_j= 1$ and since also  $a_k\le 1$, this again contradicts $\sum_{i\in \mathbb{N}_5} a_i > 2$.
    This leads us to the condition $\sum_{i\in S} a_i = 1$ for $S\in\mathbb{S}$.
 There are $\binom{5}{3}=10$ such subsets $S \subset \mathbb{N}_5$, that is $|\mathbb{S}|$ = 10.
Lemma \ref{lem:combinatorial computation of I} implies that for a given weight $\A\in\mathcal{D}_{0,5}$ one always has $d(\A)\le 4$.

The classification into types (A)-(F) in Definition \ref{defn:classification for n=5} corresponds to the number and configuration of subsets $I \in \mathbb{I}$ such that $\A(I^c) \le 1$.
On each connected component, the sign $(\sum_{j \in S} a_j) -1$ is constant and conversely, the chambers are distinguished by the collection of these signs.
Item (1) is exactly the case $(\sum_{j\in S} a_j) -1 >0$ for all $S$. This shows the claim combined with the description of the coarse chamber decomposition (\ref{eq:coarse chamber decomposition}).
Next, item (2) is a consequence of $|\mathbb{S}|=10$.
We thus move to (3) and type (C).
The problem is counting the number of $I_i \in\mathbb{I}$ with $\A(I_i^c)\le 1$ for $i=1,2$.
If $|I_1^c\cup I_2^c| = \mathbb{N}_5$, then this implies $\sum_{i=1}^5 a_i \le \A(I_1^c) + \A(I_2^c) \le 2$.
Hence, we must have $|I_1^c \cup I_2^c| = 4$, which implies that the number of possible cases equals $\binom{5}{3} \times \binom{3}{2}\times \binom{2}{1}/2 = 30$.
For the case of $d(\A) = 3$, let $D(\A) = \{I_1,I_2,I_3\}$.
If $I_1^c\cap I_2^c\cap I_3^c = \emptyset$, then we can find indices $i,j$ with $I_i^c\cup I_j^c = \mathbb{N}_5$.
This contradicts the definition of weight as above.
Hence, we have $|I_1^c\cap I_2^c \cap I_3^c| =1$ or $2$ since $|I_1^c\cap I_2^c \cap I_3^c| \le 2$.
The number claimed now follows from a direct computation.
Item (6) corresponds to the choice of $I_i\in\mathbb{I}$ with $D(\A) = \{I_1,\cdots, I_4\}$ and $|I_1^c\cup I_2^c\cup I_3^c\cup I_4^c| = 4$.
The number of possible choices is $\binom{5}{4} = 5$.

It remains to show that the types in in Definition \ref{defn:classification for n=5} characterize the coarse chambers modulo the $S_5$-action. Clearly, different types are not in the same $S_5$-orbit.
To see that all weights $\A$ of the same type are $S_5$-equivalent, it is helpful to invoke the complete graph $K_5$ on the five vertices $\{1, \ldots , 5 \}$. Its 10 edges are in $1:1$ correspondence 
with the set $\mathbb I$ defined in (\ref{equ:defI}). Now, given a weight $\A$, the set $D(\A)$ defines a subgraph of $K_5$ with $d(\A)$ edges. These subgraphs are depicted in 
Figure \ref{fig:k5-subgraphs}
under the assumption that $a_1 \leq a_2 \leq \ldots \leq a_5$. 
It is obvious that two subgraphs 
defined by weights $\A$ and $\A'$ are isomorphic as subgraphs of $K_5$, if and only if the weights are of the same type. 
But now the claim follows since $\Aut(K_5)\cong S_5$.

\begin{figure}[t]
  \centering
  \setlength{\tabcolsep}{2pt}\renewcommand{\arraystretch}{0}

  \begin{subfigure}[t]{.31\columnwidth}
    \centering
    \includegraphics[width=\linewidth,trim=2pt 2pt 2pt 2pt,clip]{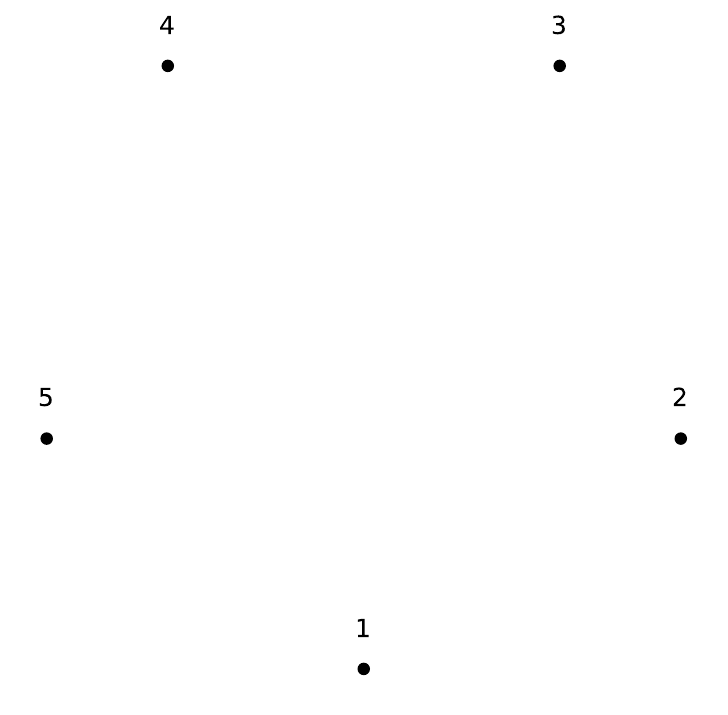}
    \caption{Type (A)}
  \end{subfigure}\hfill
  \begin{subfigure}[t]{.31\columnwidth}
    \centering
    \includegraphics[width=\linewidth,trim=2pt 2pt 2pt 2pt,clip]{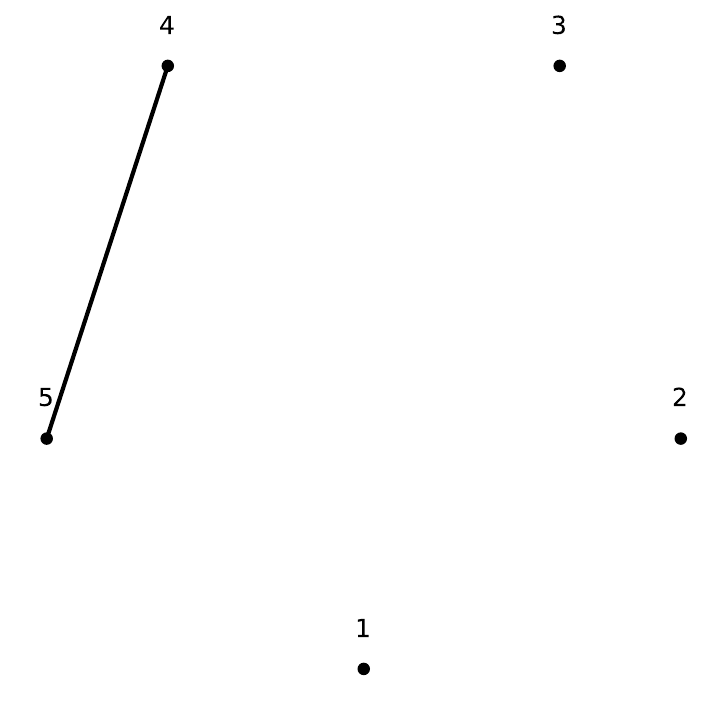}
    \caption{Type (B)}
  \end{subfigure}\hfill
  \begin{subfigure}[t]{.31\columnwidth}
    \centering
    \includegraphics[width=\linewidth,trim=2pt 2pt 2pt 2pt,clip]{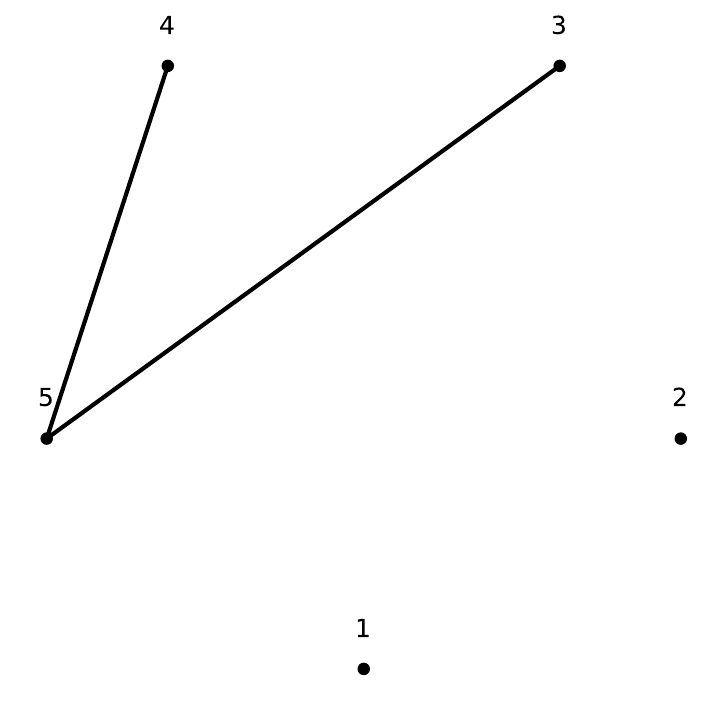}
    \caption{Type (C)}
  \end{subfigure}

  \vspace{0.4em}

  \begin{subfigure}[t]{.31\columnwidth}
    \centering
    \includegraphics[width=\linewidth,trim=2pt 2pt 2pt 2pt,clip]{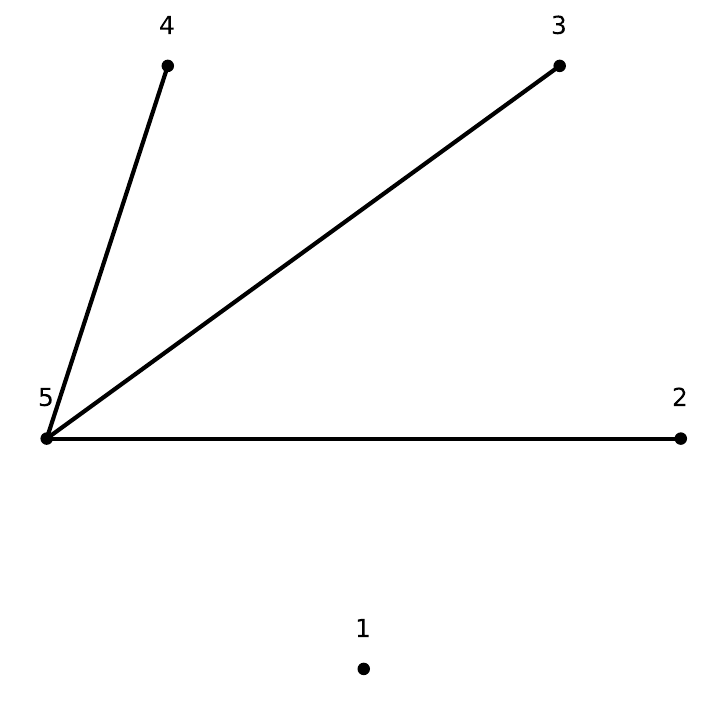}
    \caption{Type (D)}
  \end{subfigure}\hfill
  \begin{subfigure}[t]{.31\columnwidth}
    \centering
    \includegraphics[width=\linewidth,trim=2pt 2pt 2pt 2pt,clip]{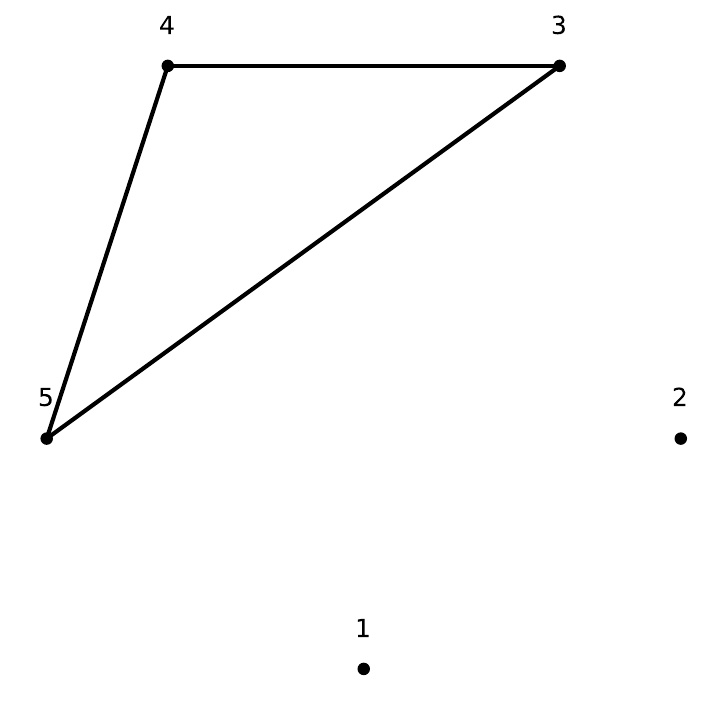}
    \caption{Type (E)}
  \end{subfigure}\hfill
  \begin{subfigure}[t]{.31\columnwidth}
    \centering
    \includegraphics[width=\linewidth,trim=2pt 2pt 2pt 2pt,clip]{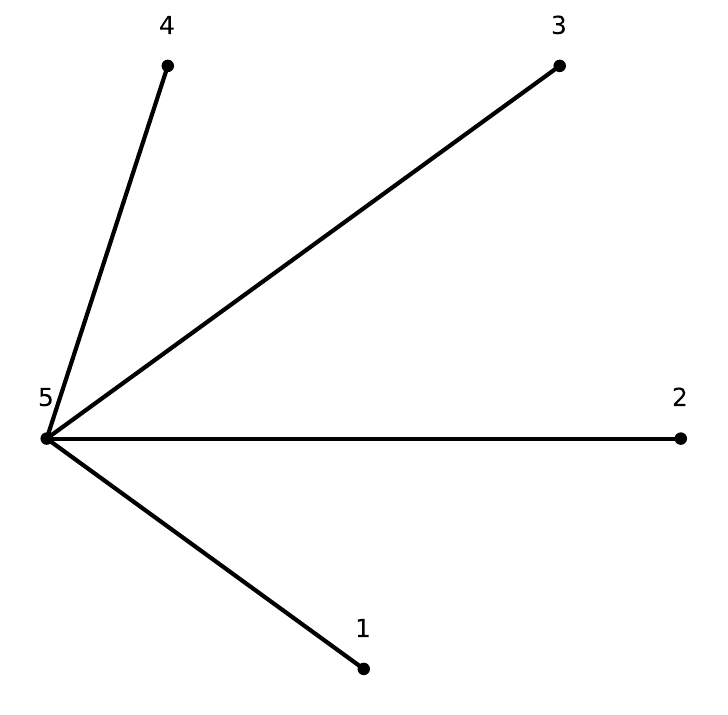}
    \caption{Type (F)}
  \end{subfigure}

  \caption{Subgraphs of $K_5$ representing the six $S_5$-orbits of coarse chambers.}
  \label{fig:k5-subgraphs}
\end{figure}

\end{proof}

Note that there are $10$ divisors $D_{I,I^c}$ in $\overline{\M}_{0,5}$, corresponding to $I\in\mathbb{I}$.
We now turn to the geometry of the moduli spaces $\overline{\M}_{0,\A}$. For a general discussion, see also \cite[Subsection 6.2]{Has03}.
\begin{defn}
    Let $\A,\mathcal{B}\in\mathcal{D}_{0,5}$ be weights.
    We say that $\A \ge \mathcal{B}$ if $\A(I^c) \le 1$ implies $\mathcal{B}(I^c) \le 1$ for any $I\in\mathbb{I}$.
    Furthermore, we say that $\A > \mathcal{B}$ if $\A \ge \mathcal{B}$ and there exists an $I\in\mathbb{I}$ with $\A(I^c)>1$ and $\mathcal{B}(I^c) \le 1$.
\end{defn}

The following proposition gives a geometric interpretation of $\overline{\M}_{0,\A}$. We will summarize the properties of these surfaces in Table \ref{tab:six-types} below.
    \begin{prop}
    \label{prop:contraction}
    Let $\A\in\mathcal{D}_{0,5}$ be a weight.
   \begin{enumerate}
\item The moduli spaces $\overline{\M}_{0,\A}$ are isomorphic to a del Pezzo surfaces of degree $5+d(\A)$.
       \item If $\A$ is of type \textit{(D)}, then $\overline{\M}_{0,\A}$ is the Hirzebruch surface $\mathbb{F}_1$, that is a blowup of $\P^2$ at one point.
       \item If $\A$ is of type \textit{(E)}, then $\overline{\M}_{0,\A}$ is $\P^1\times \P^1$.
   \item Let $\mathcal{B}\in\mathcal{D}_{0,5}$ be a weight such that $\A \ge \mathcal{B}$.
   The reduction map $\rho_{\mathcal{B},\A}:\overline{\M}_{0,\A} \to \overline{\M}_{0,\mathcal{B}}$, defined in \cite[Theorem 4.1]{Has03}, contracts those $(-1)$-curves $D_{I,I^c}$ for which  $\A(I^c)>1$ and $\mathcal{B}(I^c)\le 1$.
   \end{enumerate}

    \end{prop}
    \begin{proof}
     The moduli space $\overline{\M}_{0,5}$, which is the case of type (A), is the blowup of $\P^2$ in 4 general points, that is, the del Pezzo surface of degree 5 by \cite[Remark 3.7]{KM13}; 
     see also \cite[Subsection 6.1]{Has03} or \cite{Kap93a}.  
On $\overline{\M}_{0,5}$, we have 10 curves $D_{I,I^c}$, the vital divisors in the sense of \cite{KM13}.
It follows from the proof of \cite[Lemma 4.3]{KM13} that the self-intersection number of these vital curves is $-1$. Indeed, these are the 10 negative curves on the del Pezzo surface of degree $5$; see also \cite[Proposition 6.5]{Kon07}.
Moving into a chamber with $\sum_{i \in I^c} a_i < 1$ means contracting a $(-1)$-curve $D_{I,I^c}$.  

For a systematic discussion, we divide $\mathbb{S}$ into the following two subsets:
\begin{align*}
    \mathbb{S}_1\defeq\{S\in\mathbb{S}\mid 5\not\in S\},\ \mathbb{S}_2\defeq\{S\in\mathbb{S}\mid 5\in S\}.
\end{align*}
Note that $|\mathbb{S}_1| = 4$ and $|\mathbb{S}_2| = 6$.
Since $D_{I,I^c}$ and $D_{J,J^c}$ intersect for $I,J\in\mathbb{I}$ if and only if $I^c \supset J$, one can check that $D_{I,I^c}$ and $D_{J,J^c}$ do not intersect for $I^c, J^c\in\mathbb{S}_1$, and for any $I^c\in\mathbb{S}_1$, there are three elements in $\mathbb{S}_2$ whose corresponding vital curves intersect with $D_{I,I^c}$.
This recovers the well known configuration of the 10 negative curves on a del Pezzo surface of degree 5, see \cite[Ch. IV. \S 26.9]{Man86}.  
Each time one increases $d(\mathcal A)$ by 1, the number of $(-1)$-curves decreases. In the first two steps, this leads to the unique (as abstract surfaces) del Pezzo surfaces of degree 6 and 7, respectively, but there are several ways of blowing down curves  $D_{I,I^c}$. For example, in the first step, there is a choice of 10 curves corresponding to the 10 chambers of type (B). 
Somewhat more care is needed for the third step as there are two different del Pezzo surfaces of degree 8.
We shall first assume that $\A$ is of type (D).
This means that the reduction map $\overline{\M}_{0,5} \to \overline{\M}_{0,\A}$ 
contracts three curves which are indexed by $I_i$ for $i=1,2,3$, with $|I_1\cup I_2 \cup I_3| = 4$ by Definition \ref{defn:classification for n=5} (D).
Hence, we can take $I_4\in\mathbb{I}\setminus \{I_1,I_2,I_3\}$ such that $|I_1\cap I_2 \cap I_3 \cap I_4| = 1$.
Thus there exists another weight $\mathcal{B}\in\mathcal{D}_{0,5}$ with $\A > \mathcal{B}$, $\mathcal{A}(I_4^c)> 1$
and $\mathcal{B}(I_4^c)\le 1$, which induces a contraction morphism $\overline{\M}_{0,\A} \to \overline{\M}_{0,\mathcal{B}}$.
This gives a contraction of a del Pezzo surface of degree 8 to a del Pezzo surface of degree 9 (which must then necessarily be $\P^2$). But this implies that $\overline{\M}_{0,\A}$ is isomorphic to $\mathbb{F}_1$ as there is no 
morphism from $\P^1 \times \P^1$ to $\P^2$.
Finally, Hassett \cite[Remark 8.4]{Has03} shows that type (E) is isomorphic to $\P^1\times \P^1$.
This concludes the proof.

In  Table \ref{tab:six-types} below, we summarize the properties of the six cases which occur, up to the $S_5$-action. 
In each case, the listed weight vector is only a representative, and the contracted boundary curves are exactly those divisors $D_{I,I^c}$ for which $\A(I^c)\le 1$.
Concerning the connection to the del Pezzo surface of degree 6, we also refer the reader to \cite[Remark 0.12]{MM14} and \cite[Remark 1.6]{MM17}.
\end{proof}
\begin{table}[htbp]
\centering
\small
\begin{tabular}{c|c|l|p{7.8cm}}
Type & $d(\A)$ & Representative weight $\A=(a_1,\dots,a_5)$ & Contracted boundary curves \\
\hline
(A) & $0$ & $(1/2,1/2,1/2,1/2,1/2)$ & none \\[2pt]

(B) & $1$ & $(1/4,1/4,1/4,3/4,3/4)$
& $D_{\{4,5\},\{1,2,3\}}$ \\[2pt]

(C) & $2$ & $(1/5,1/5,1/2,1/2,7/10)$
& $D_{\{4,5\},\{1,2,3\}},\ D_{\{3,5\},\{1,2,4\}}$ \\[2pt]

(D) & $3$ & $(3/20,3/10,3/10,1/2,17/20)$
& $D_{\{4,5\},\{1,2,3\}},\ D_{\{3,5\},\{1,2,4\}},$
\\
& & &
$D_{\{2,5\},\{1,3,4\}}$ \\[2pt]

(E) & $3$ & $(3/25,11/50,11/20,11/20,3/5)$
& $D_{\{4,5\},\{1,2,3\}},\ D_{\{3,5\},\{1,2,4\}},$
\\
& & &
$D_{\{3,4\},\{1,2,5\}}$ \\[2pt]

(F) & $4$ & $(3/10,3/10,3/10,3/10,17/20)$
& $D_{\{4,5\},\{1,2,3\}},\ D_{\{3,5\},\{1,2,4\}},$
\\
& & &
$D_{\{2,5\},\{1,3,4\}},\ D_{\{1,5\},\{2,3,4\}}$
\end{tabular}
\caption{The six coarse chamber types in $D_{0,5}$ up to the $S_5$-action. Here $d(\A)=|D(\A)|$, and $\deg(\overline{\M}_{0,\A}) =5+d(\A)$ by Proposition~2.7.}
\label{tab:six-types}
\end{table}

The theory developed by Deligne and Mostow \cite{DM86} addresses when moduli spaces of weighted configurations 
of (partially) ordered points on $\mathbb{P}^1$ admit a complex hyperbolic structure. This means that they can be realized as a quotient of a complex ball by a discrete group of isometries, via the monodromy of Appell-Lauricella hypergeometric functions. 
In this situation the Baily-Borel compactifications of these ball quotients are isomorphic to GIT quotients of $(\mathbb P^1)^n$ where the weights determine the ample line bundle $\L$ on $(\mathbb P^1)^n$ and its 
$\SL_2(\C)$-linearization.  
For the Baily-Borel compactifications, see \cite{BB66} for arithmetic subgroups and \cite{Mok12} for non-arithmetic subgroups; see also \cite[Section 8]{YZ24}.

Kirwan, Lee, and Weintraub showed that all del Pezzo surfaces of degree $\ge 5$ arise as Deligne-Mostow ball quotients \cite[Theorem 4.1 (i)-(vi)]{KLW87}.
Combined with Proposition \ref{prop:contraction}, we can observe that $\overline{\M}_{0,\A}$ are isomorphic to ball quotients as abstract varieties.
We furthermore show that these isomorphisms preserve the moduli problems.
To make this precise, we formalize what it means for such a modular interpretation to exist:
\begin{defn}
\label{defn:modular period}
    Let $\A\in\mathcal{D}_{0,5}$.
    We say that $\overline{\M}_{0,\A}$ \emph{has a Deligne-Mostow modular interpretation} if it is isomorphic to the Baily-Borel compactification of a ball quotient $\overline{\B^2/\Gamma}$ for a (not necessarily arithmetic) subgroup $\Gamma\subset\U(1,2)$ and the isomorphism is given by a composition of the following two isomorphisms:
    \begin{enumerate}
        \item (A variation of weights in the Hassett spaces and GIT quotients.) An isomorphism $\overline{\M}_{0,\A} \cong (\P^1)^5/\!/_{\L}\SL_2(\C)$ given by the limit of configurations, that is either of the morphisms defined in \cite[Theorems 8.2, 8.3]{Has03} for some $\L$.
        \item (A Deligne-Mostow uniformization.) The Delinge-Mostow isomorphism of the varieties $(\P^1)^5/\!/_{\L}\SL_2(\C)\cong \overline{\B^2/\Gamma}$ given by the INT condition \cite{DM86} (and not the $\Sigma$INT condition introduced in \cite{Mos86}).
    \end{enumerate}
\end{defn}
Starting from moduli spaces of weighted pointed rational curves, we can establish the following connection to ball quotients.
\begin{prop}
\label{prop:relation to DM}
For any weight $\A\in\mathcal{D}_{0,5}$ the moduli space $\overline{\M}_{0,\A}$ has a Deligne-Mostow modular interpretation. All cases but (E) can be realized by a compact ball quotient.
\end{prop}
\begin{proof}
 We now prove that all Hassett spaces $\overline{\M}_{0,\A}$ are the Baily-Borel compactifications of ball quotients, preserving their modular interpretation as configuration spaces. 
The starting point is  \cite[Theorems 8.2, 8.3]{Has03}, which relates the 
Hassett spaces to GIT quotients of the form $(\P^1)^5/\!/_{\L}\SL_2(\C)$. 
Here $\L$ is an ample line bundle on $(\P^1)^5$ which is $\SL_2(\C)$-linearized.

To start with,  
we consider the typical case (\cite[Theorem 8.2]{Has03}), which means 
that there are no properly semi-stable points.  
 By Hassett's theorem, there exists a 
 weight $\A\in\mathcal{D}_{0,5}$ and an isomorphism
 \[(\P^1)^5/\!/_{\L}\SL_2(\C) \cong \overline{\M}_{0,\A}\]
 for some typical line bundle $\L$.
 The appendix \cite[Appendix]{Thu98} enumerates all cases in which such a GIT quotient $(\P^1)^5/\!/_{\L}\SL_2(\C)$ is isomorphic to a compact 
ball quotient $\B^2/\Gamma$ (via the Deligne-Mostow isomorphism for a subgroup $\Gamma \subset \U(1,2)$).
To start with, we focus on ``pure" and ``compact" cases in the notation \cite[Appendix]{Thu98}.
The notation ``pure" refers to the INT condition introduced in \cite{DM86}, not the $\Sigma$INT condition in \cite{Mos86}.
This implies that $(\P^1)^5/\!/_{\L}\SL_2(\C)$ itself (not divided by a symmetric group) admits a ball quotient structure.
Also, since the properly semi-stable points on $(\P^1)^5/\!/_{\L}\SL_2(\C)$ corresponds to the cusps of the Baily-Borel compactification $\overline{\B^2/\Gamma}$, the condition that $\L$ is typical implies that $\B^2/\Gamma$ must be compact.
The chamber decomposition from \cite[Proposition 5.1]{Has03} determines the isomorphism class of such ball quotients from the weight data.
Below, we follow the numbering (No.) in \cite[Appendix]{Thu98}. 
For a given $\L=\OO(r_1,_\cdots,r_5)$, normalized by $\sum_i r_i =2 $, counting the number of $S\in\mathbb{S}$ such that $\sum_{i\in S} r_i < 1$,  
we can deduce from \cite[Theorem 8.2]{Has03} that the Deligne-Mostow ball quotient associated with either of No. 9, No. 47, No. 48, No. 72, No. 74, No. 75, No. 78, or No. 89 is isomorphic to type (A) moduli spaces preserving their moduli interpretations as weighted configuration spaces.
    Similarly, No. 46, No. 65, No. 69, or No. 85 give type (B), No. 70 gives type (C), No. 45, No. 49, No. 57, or No. 79 give type (D), and No. 68 gives type (F).

    The remaining case is type (E). For this, we consider the atypical case.
    Let $\A\in\mathcal{D}_{0,5}$ be a weight such that $\A(I_i^c)\leq 1$ for $I_1=\{1,2\}, I_2=\{1,3\}, I_3=\{2,3\}$; the other cases are similar. Following \cite[Theorem 4.1 (iv)]{KLW87}, we consider the  
    atypical line bundle $\L=\OO(1/2,1/2,1/2,1/4,1/4)$, \cite[Theorem 8.3]{Has03} gives rise to a birational morphism $\rho:\overline{\M}_{0,\A} \to (\P^1)^5/\!/_{\L}\SL_2(\C)$.
    By Proposition \ref{prop:contraction} (3), the moduli space $\overline{\M}_{0,\A}$ is isomorphic to $\P^1\times\P^1$. As this surface is minimal, it follows that
    $(\P^1)^5/\!/_{\L}\SL_2(\C)$ is also isomorphic to $\P^1\times\P^1$ and that $\rho$ is an isomorphism.
    Now, No. 8 in \cite[Appendix]{Thu98} shows that $(\P^1)^5/\!/_{\L}\SL_2(\C)$ is isomorphic to the Baily-Borel compactification of a non-compact ball quotient through the Deligne-Mostow morphism.
    This concludes the proof. There are other cases, which we will not list here, which lead to type (E). They can be treated in the same way. A direct inspection further shows that there are no compact Deligne-Mostow quotients 
    giving $\P^1 \times \P^1$.
   
In Table \ref{tab:thurston-types} we summarize the correspondence between the chamber types in Definition \ref{defn:classification for n=5} and the numbering in \cite[Appendix]{Thu98}. 
We remark that the original list \cite[p. 86]{DM86} and, accordingly, also the tables
\cite[Tables I, II]{KLW87} are incomplete (for example, in the case of type (A), No. 72 and No. 78  are missing). 
      
\end{proof}

    \begin{table}[htbp]
\centering
\small
\begin{tabular}{c|l}
Type of $\overline{\M}_{0,\A}$ & No.\ in \cite[Appendix]{Thu98} \\
\hline
(A) & 9, 47, 48, 72, 74, 75, 78, 89 \\
(B) & 46, 65, 69, 85 \\
(C) & 70 \\
(D) & 45, 49, 57, 79 \\
(E) & 8 and further atypical cases of the same form \\
(F) & 68
\end{tabular}
\caption{Correspondence between the chamber types in Definition \ref{defn:classification for n=5} and the numbering in \cite[Appendix]{Thu98} appearing in the proof of Proposition \ref{prop:relation to DM}. The cases (A), (B), (C), (D), and (F) arise from pure compact cases, while type (E) is obtained from atypical cases.}
\label{tab:thurston-types}
\end{table}

We also note that there is another approach to the observation that $\overline{\M}_{0,5}$ is isomorphic to a 2-dimensional compact ball quotient. 
This was pursued by Kond\=o in \cite{Kon07} studying  moduli spaces of K3 surfaces with an order 5 automorphism.
Based on his observation, Kond\=o posed the problem whether all varieties appearing in the Deligne–Mostow list, are related to periods of K3 surfaces  \cite[Problem 7.2]{Kon07}, which was later answered 
positively by Moonen \cite[Theorem A]{Moo18}.
Our previous discussion shows that any moduli space $\overline{\M}_{0,\A}$ has a ball quotient model for $n=5$.
Also, in higher-dimensional cases, there is the possibility that there are other period maps, different from that coming from the monodromy of Appell-Lauricella hypergeometric functions, which lead to ball quotients. This leads to the following questions.
\begin{que}
\label{que:ball1}
    Given a weight $\A\in\mathcal{D}_{0,n}$, is there always a suitable period map such that the moduli space $\overline{\M}_{0,\A}$ admits a ball quotient model?
\end{que}
In certain instances, ball quotients are known to admit more than one modular interpretation. This phenomenon is sometimes referred to as {\textit {Janus-like}} algebraic varieties; see \cite{HW94}. This leads naturally to the following question:  
\begin{que}
\label{que:ball2}
Let $\Gamma\subset\U(1,n-3)$ be a discrete subgroup and $\M$ be a moduli space which admits an open immersion $\M\hookrightarrow\B^{n-3}/\Gamma$. Is there some suitable weight 
$\A\in\mathcal{D}_{0,n}$ with an open immersion  $\M_{0,n}\hookrightarrow \B^{n-3}/\Gamma$ such that $\overline{\M}_{0,\A}\cong \overline{\B^{n-3}/\Gamma}\cong \overline{\M}$?
If so, can one give a geometric explanation for such an isomorphism?
\end{que}
There is an extensive list of papers on moduli spaces that admit ball quotient models, including the following references: 
\cite{ACT02, ACT11, CMGHL23,CMJL12, DM86, DvGK05, GKS21, HKM24,HM25a,HM25b, HW94, KLW87, Kon07, Moo18, Mos86}. The moduli problems addressed in these works encompass configurations of weighted points, cubic surfaces, cubic threefolds, non-hyperellitpic curves of genus 4,
and various moduli spaces of K3 surfaces (with prescribed automorphism group).  
We note that this list is not exhaustive.

\section{Fulton's conjecture}
\label{sec:Fulton's conjecture}
The divisors $D_{I,I^c}$ for any $I\subset \mathbb{N}_n$ with $2\leq |I|\leq \lfloor n/2 \rfloor$ 
on $\overline{\M}_{0,n}$ are called \textit{vital divisors}.
A complete intersection of them is  called a \textit{vital cycle}.
It is known that the vital curves correspond to partitions of $\mathbb{N}_n$ into four disjoint subsets.  If the cardinalities of these subsets are 
$a, b, c$ and $d$, we denote the corresponding curve by $C(a,b,c,d)$; see \cite[Section 4]{KM13}.
\begin{defn}
    Let $E$ be a divisor on $\overline{\M}_{0,n}$.
    We say $E$ is \textit{$F$-nef} if $C(a,b,c,d).E\geq 0$ for all partitions corresponding to the numbers $a,b,c$ and $d$.
\end{defn}
Fulton's conjecture \cite[Conjecture 0.2]{GKM02}
asserts that $F$-nef divisors are nef:
\begin{conj}[{Fulton's conjecture}]
\label{conj:fulton}
    The $F$-nef divisors are nef, that is, if $C(a,b,c,d).E\geq 0$ for all $(a,b,c,d)$, then $C.E\geq 0$ for all irreducible curves $C$.
\end{conj}

This conjecture has attracted considerable attention. It was Gibney-Keel-Morrison and Keel-McKernan who solved it in the special case  $n=5$; see also Remark \ref{rem:F conjecture}.
\begin{thm}[{A special case of \cite[Corollary 0.4]{GKM02}, \cite[Theorem 1.2]{KM13}}]
\label{thm:fulton}
Conjecture \ref{conj:fulton} holds for $\overline{\M}_{0,5}$.
\end{thm}

The case $n=5$ is particularly accessible because $\overline{\M}_{0,5}$ is a smooth projective surface, in fact the del Pezzo surface of degree $5$. Thus, questions about nefness can be reduced to intersection theory on curves, and the relevant geometry can be described explicitly in terms of the ten boundary $(-1)$-curves and their contractions. Moreover, the vital curves appearing in Fulton's conjecture are easy to 
understand in this surface case, so that the passage from $F$-nefness to nefness becomes especially effective. This is precisely the feature that makes the arguments in Section \ref{section: proof of the main theorem} concrete: once the divisor is written as a pullback from $\overline{\M}_{0,\A}$ plus an exceptional part, one can verify positivity by explicit intersection computations on a del Pezzo surface.

\begin{rem}
\label{rem:F conjecture}
    Here we summarise the present knowledge on the various Fulton's conjectures for the cone of curves on the moduli space $\overline{\M}_{0,n}$ for general $n$.
    \begin{enumerate}
        \item Conjecture \ref{conj:fulton} is also called the \textit{weak} version.
        It is proved by Gibney-Keel-Morrison \cite[Corollary 0.4]{GKM02} and Keel-McKernan \cite[Theorem 1.2]{KM13} for $n\leq 7$.
        For the case $n \geq 8$, no counter‑example is known, but neither is a proof.   
        \item The \textit{strong} Fulton conjecture asserts that any $F$-nef divisor can be written as a summation of effective boundary divisors.
        This is proven for $n\leq 6$ by \cite[Theorem 2]{FG03}, and $n=7$ by Larsen \cite[Theorem 3.2]{Lar12}.
        However, Paxton \cite[Proposition 1]{Pix13} constructed an $F$‑nef, base‑point‑free divisor on $\overline{\M}_{0,12}$ that is not numerically equivalent to any effective boundary combination, thereby giving a counter‑example to the strong Fulton conjecture for $n=12$ (and, by pull‑back, for all $n\ge 12$).
        \item The bridge theorem \cite[Theorem 0.3]{GKM02} reduces Fulton's conjecture for $\overline{\M}_{g,n}$ to the \textit{symmetric} Fulton's conjecture for $\overline{\M}_{0,n}$.
        Fedorchuk \cite[Corollary 3]{Fed20} proved the symmetric version for $n \leq 35$.
    \end{enumerate}
\end{rem}

\section{Proof of Theorem \ref{thm:n=5}}
\label{section: proof of the main theorem}
Our proof of Theorem \ref{thm:n=5} is based on Fulton's conjecture, which is known to be true in this case (Theorem \ref{thm:fulton}).
Hence, below, we compute the intersection numbers of $D_{I,I^c}$ and the curves $C(a,b,c,d)$.

The proof proceeds in three steps. First, for a given weight $\A\in \mathcal{D}_{0,5}$, we rewrite the divisor
\[
K_{\overline{\M}_{0,5}}+\sum_{I\in\mathbb{I}} d_{I,I^c}D_{I,I^c}
\]
as the pullback of a divisor on $\overline{\M}_{0,\A}$ plus an effective $\rho_{\A}$-exceptional divisor, where $\rho_{\A}\colon \overline{\M}_{0,5}\to \overline{\M}_{0,\A}$ is the reduction morphism. Second, we show that the divisor downstairs is $F$-nef by checking its intersection with the vital curves. Since Fulton’s conjecture holds for $n=5$, this implies that the divisor is in fact nef. Finally, we apply the basepoint-free theorem to conclude that the divisor on $\overline{\M}_{0,\A}$ is semiample, and hence the associated log canonical model is exactly $\overline{\M}_{0,\A}$.

Assume $\A\in\mathcal{D}_{0,5}$ is not of type (A).
We consider the divisors $E\defeq \sum_{J\in\mathbb{I}} D_{J,J^c}$ and $F\defeq E - \sum_{I\in D(\A)} D_{I,I^c}$ on $\overline{\M}_{0,5}$.
Let $\rho_{\A}:\overline{\M}_{0,5} \to \overline{\M}_{0,\A}$ be the contraction map constructed in \cite[Theorem 4.1]{Has03}.
We set $D_{J,J^c}^{\A}\defeq \rho_{\A*} D_{J,J^c}$ for any $J\in\mathbb{I}\setminus D(\A)$, and $F^{\A}\defeq \rho_{\A*}F$.
We denote the centre of the blowup $\rho_{\A}$ by $Z^{\A}$.
A straightforward computation shows
\[\rho_{\A}^*F^{\A} = F + \sum_{I\in D(\A)}(\ord_{Z^{\A}}F^{\A})D_{I,I^c} = F + 3\sum_{I\in D(\A)} D_{I,I^c}.\]
The last equation follows from counting the number of $J\in\mathbb{I}\setminus D(\A)$ such that $J\subset I^c$ for all $I\in D(\A)$, see \cite[Introduction]{Kee92} and also the proof of \cite[Theorem 2.3.4]{Sim08}.
Since $\mathrm{codim}_{\overline{\M}_{0,\A}} Z^{\A} = 2$ and $\rho_{\A}$ is the blowup of a smooth variety in a smooth centre by Proposition \ref{prop:contraction},
the general formula for the canonical divisor of a 
blowup gives
\[K_{\overline{\M}_{0,5}} = \rho_{\A}^*K_{\overline{\M}_{0,\A}} + \sum_{I\in D(\A)} D_{I,I^c}.\]
These two relationships show that for any $\alpha_I,\beta \in\Q$ 
we have
\begin{equation} \label{equ:K05}
K_{\overline{\M}_{0,5}} + \sum_{I\in D(\A)} \alpha_I D_{I,I^c} + \beta F = \rho_{\A}^* (K_{\overline{\M}_{0,\A}} + \beta F^{\A}) + \sum_{I\in D(\A)}(\alpha_I - 3\beta + 1) D_{I,I^c}. 
\end{equation}

Defining 
\[\L(\beta) \defeq K_{\overline{\M}_{0,5}} + (3\beta -1)\sum_{I\in D(\A)}D_{I,I^c} + \beta F,\]
it follows from (\ref{equ:K05}) that $\L(\beta) = \rho_{\A}^*(K_{\overline{\M}_{0,\A}} + \beta F^{\A})$.
Furthermore, the 
$\Q$-divisor \[(K_{\overline{\M}_{0,5}} + \sum_{I\in D(\A)}\alpha_I D_{I,I^c} + \beta F) - \L(\beta) = \sum_{I\in D(\A)}(\alpha_I -3\beta +1)D_{I,I^c}\] is effective and $\rho_{\A}$-exceptional if  $\alpha_I - 3\beta + 1 \geq 0$ for any $I\in D(\A)$.

 We shall compute the intersection number with vital curves.
    Here, we have to consider all vital curves, not only \textit{symmetric} ones, because our divisors are not symmetric.
    In this case, the representatives of vital curves correspond to \begin{align*}
        (a,b,c,d) = (1,1,1,2), (1,1,2,1), (1,2,1,1),(2,1,1,1).
    \end{align*}
    By \cite[Lemma 2.3.1]{Sim08}, we have 
    \[\L(\b) = 3(\b-1/2)\sum_{I\in D(\A)} D_{I,I^c} + (\b-1/2)F.\]
    Given this description, we can use \cite[Lemma 4.3]{KM13} to compute the intersection numbers of $\L(\b)$ with the curves $C(a,b,c,d)$.
    Here we treat the computation for the case when $\A$ is of type (B), $D(\A) = \{I\}$, and  $I = \{4,5\}$ in detail. The other cases are similar and will be left to the reader.
    For simplicity we put $T_1\cup T_2\cup T_3\cup T_4$ as the partition corresponding to $(a,b,c,d)$.
The theorem by Keel and Mckernan \cite[Lemma 4.3]{KM13} tells us that
\begin{align}
\label{eq:intersection 1112}
    \L(\b)\cdot C(1,1,1,2) = 3(\b-1/2)\cdot \gamma_1 + (\b-1/2)\cdot \sum_J\gamma_{2,J}
\end{align}
where 
\begin{align*}
    \gamma_1 &=\#\{i\in\{2,3,4\}\mid I=T_1\cup T_i\ \mathrm{or}\ I^c=T_1\cup T_i\} - \#\{i\in\mathbb{N}_4\mid I=T_i\ \mathrm{or}\ I^c=T_i\}\\
    \gamma_{2,J} &=\#\{i\in\{2,3,4\}\mid J=T_1\cup T_i\ \mathrm{or}\ J^c=T_1\cup T_i\} - \#\{i\in\mathbb{N}_4\mid J=T_i\ \mathrm{or}\ J^c=T_i\}
\end{align*}
for $J\neq I$.

We consider the intersection with $C(1,1,1,2)$, which corresponds to the partition $\{1\}\cup\{2\}\cup\{3\}\cup\{4,5\}$.
Then, we have $\gamma_1 = 0-1 = -1$ and $\sum_J\gamma_{2,J} = 3-0 = 3$.
This computation leads us to rewrite (\ref{eq:intersection 1112}) as 
\[\L(\b)\cdot C(1,1,1,2) = -3(\b-1/2) + 3(\b-1/2) = 0.\]
Note that this shows that $C(1,1,1,2)$ is $\rho_{\A}$-exceptional and coincides with the exceptional divisor $D_{I,I^c}$.
For other quadruples, a similar computation shows that $\L(\b)\cdot C(a,b,c,d)\ge 0$ if and only if $\beta \geq 1/2$.
In this range, $\L(\beta)$ is $F$-nef.
    By Fulton's conjecture (Theorem \ref{thm:fulton}), $\L(\beta)$ is nef.
    Moreover if $\beta > 1/2$, then the intersection number of $\L(\beta)$ to $C(a,b,c,d)$ is 0 if and only if $C(a,b,c,d)$ is $\rho_{\A}$-exceptional.    
A similar computation of the intersections numbers applies to the other types from (C) to (F).
The key point is that for any $\A$, we can find a partition $(a,b,c,d)$ such that the curve $C(a,b,c,d)$ is not contracted by $\rho_{\A}$.
Computing the intersection of $\L(\b)$ with such $C(a,b,c,d)$, combined with Fulton's conjecture (Theorem \ref{thm:fulton}), we conclude that if $\beta-1/2>0$, then $\L(\beta)$ is nef, and the intersection number is zero if and only if $C(a,b,c,d)$ is $\rho_{\A}$-exceptional.

    We can now conclude the proof by \cite[Corollary 2.3.5]{Sim08}, which, for the reader's sake, we give below.
    Summarizing, we have 
    \begin{align}
        H^0(\overline{\M}_{0,5}, K_{\overline{\M}_{0,5}} + \sum_{I\in D(\A)}\alpha_I D_{I,I^c} + \beta F) &= H^0(\overline{\M}_{0,5}, \L(\beta) + \sum_{I\in D(\A)}(\alpha_I - 3\beta + 1) D_{I,I^c}) \label{eq:1}\\
        &=H^0(\overline{\M}_{0,\A}, K_{\overline{\M}_{0,\A}} + \beta F^{\A})\notag
    \end{align} 
    by \cite[Proposition 2.2.3]{Sim08} (apply the theorem to $D=K_{\overline{\M}_{0,\A}} + \b F^{\A}$ and the $\rho_{\A}$-exceptional divisor $F = (\a-3\b+1)D_{I,I^c}$ in his notation). 
    Also, from \cite[Lemma 2.2.2]{Sim08}, we have 
    \begin{align}
        \label{eq:2}
        H^0(\overline{\M}_{0,5},\L(\beta)) = H^0(\overline{\M}_{0,\A}, K_{\overline{\M}_{0,\A}} + \beta F^{\A}).
    \end{align}
    By \cite[Lemma 2.3.1]{Sim08}, if $\alpha_I - 3\beta + 1 >0$ and $\beta > 1/2$, then $K_{\overline{\M}_{0,5}} + \sum_{I\in D(\A)}\alpha_I D_{I,I^c} + \beta F$ is big. 
    In particular, the equation $H^0(\overline{\M}_{0,5},\L(\beta)) = H^0(\overline{\M}_{0,5}, K_{\overline{\M}_{0,5}} + \sum_{I\in D(\A)}\alpha_I D_{I,I^c} + \beta F)$ by (\ref{eq:1} and (\ref{eq:2})) implies that 
    $\L(\beta)$ is also big.
    Since bigness is an open condition, it follows from the Kawamata basepoint free theorem (\cite[Theorem 2.2.1]{Sim08}), for the pair $(\overline{\M}_{0,5}, \sum_{I\in D(\A)}\alpha_I D_{I,I^c} + \beta F)$ and the divisor $\L(\beta)$, that $\L(\beta)$ is semi-ample.
    This implies that the log canonical model $\overline{\M}_{0,5}(\{d_{I,I^c}\}_{I\in\mathbb{I}})$, where $\alpha_I \defeq d_{I,I^c}$ for $I\in D(\A)$ and $\beta\defeq d_{I,I^c}$ for $I\in\mathbb{I}\setminus D(\A)$ with $\alpha_I - 3\beta + 1 >0$ and $\beta > 1/2$, is the image of the embedding of $\overline{\M}_{0,5}$ via the linear system $|m\L(\beta)|$ for sufficiently large $m>>0$, but the latter morphism coincides with $\rho_{\A}$ since those two morphisms contract only $\rho_{\A}$-exceptional curves.
\begin{rem}[{Suggested by M. Hattori}]
\label{rem:VGIT}
In Proposition \ref{prop:relation to DM}, we studied the relationship between the moduli space of weighted pointed rational curves $\overline{\M}_{0,\A}$ and the moduli space of weighted points on rational curves $(\P^1)^5/\!/_{\L}\SL_2(\C)$.
In particular, this shows  that all GIT quotients with  typical line bundles, which correspond to inner points of the chambers in VGIT, appear as log canonical models of $\overline{\M}_{0,5}$. Here, we consider the case of atypical line bundles $\L$, lying on a wall in the weight space in VGIT.
    By \cite[Theorem 8.3]{Has03}, we obtain a birational morphism $\overline{\M}_{0,\A} \to (\P^1)^5/\!/_{\L}\SL_2(\C)$ for some $\A\in\mathcal{D}_{0,5}$.
    Take the example 
    \[\L=\OO\left(\frac{2}{3},\frac{1}{3},\frac{1}{3},\frac{1}{3},\frac{1}{3}\right).\]
The corresponding GIT quotient has, in fact, four strictly semi-stable points.
This weight 
lies on a wall in VGIT and can be considered as a limit of the following two weights:
\[\mathscr{M}_1\defeq\OO\left(\frac{2}{3}+4\epsilon,\frac{1}{3}-\epsilon,\frac{1}{3}-\epsilon,\frac{1}{3}-\epsilon,\frac{1}{3}-\epsilon\right),\ \mathscr{M}_2\defeq\OO\left(\frac{2}{3}-4\epsilon,\frac{1}{3}+\epsilon,\frac{1}{3}+\epsilon,\frac{1}{3}+\epsilon,\frac{1}{3}+\epsilon\right)\]
for small $\epsilon > 0$.
Since $\mathscr{M}_1$ and $\mathscr{M}_2$ are typical, we can take weights $\A_1, \A_2\in\mathcal{D}_{0,5}$ with $(\P^1)^5/\!/_{\L_i}\SL_2(\C) \cong \overline{\M}_{0,\A_i}$.
One can check that $\A_1$ (resp. $\A_2$) is of type (F) (resp. type (A)).
This implies that $(\P^1)^5/\!/_{\L}\SL_2(\C)$ admits a birational morphism from $\P^2$ as coarse moduli spaces, which forces the former also to be isomorphic to $\P^2$.
This space is also known to be a Deligne-Mostow variety appearing in \cite[Appendix]{Thu98} as No. 2., and \cite[Theorem 4.1 (vi)]{KLW87} gives another proof of the above statement.
Thus, in the case $n = 5$, we can analyse how the GIT quotients arise as log canonical models not only through the variation of weights in the sense of Hassett, but also through the variation of the GIT weights.
In higher dimensions, the appearance of flips necessitates a more refined analysis.
\end{rem}

\section{Log canonical models of the moduli spaces of weighted pointed rational curves with symmetric weights}
\label{sec:Log canonical models of the moduli spaces of weighted pointed rational curves with symmetric weights}
Hassett 
originally posed the problem \cite[Problem 7.1]{Has03} of determining the log canonical models for all $\overline{\M}_{0,\A}$, not only for $\overline{\M}_{0,n}$. 
As outlined in Section \ref{sec:introduction}, and discussed so far, most previous research is devoted and restricted to the case of $\overline{\M}_{0,n}$.
In this section, we discuss the Hassett spaces $\overline{\M}_{0,\A}$ that are different from the Deligne-Mumford moduli spaces $\overline{\M}_{0.n}$, 
to partially answer the original question by Hassett.
More precisely, we refine Theorem \ref{thm:lcmodel_previous} from the point of view of log canonical models for all moduli spaces $\overline{\M}_{0, n\cdot (1/k)}$ which appear in the blowup sequences
\[\overline{\M}_{0,n} \to\overline{\M}_{0,n\cdot(1/3)}\to\cdots\to\overline{\M}_{0,n\cdot(1/\lfloor (n-1)/2\rfloor)}\to (\P^1)^n/\!/\SL_2(\C),\]
see for instance \cite[Theorem 1.1]{KM11}.

Recall that $D_{I,I^c}$ denote the vital divisors on $\overline{\M}_{0,n}$ for $2\le |I|\le \lfloor n/2\rfloor$.
For a weight $\A$, let $\rho_{\A}:\overline{\M}_{0,n} \to \overline{\M}_{0,\A}$ be the reduction morphism.
If $D_{I,I^c}$ is not contracted by $\rho_{\A}$, we set $D_{I,I^c}^{\A}\defeq \rho_{\A*}D_{I,I^c}$ as in the proof in Section \ref{section: proof of the main theorem}.
Let $D^{\A}(r) \defeq \sum_{|I| = r}D_{I,I^c}^{\A}$ for $2\le r\le \lfloor n/2\rfloor$.
The following theorem, a fairly straightforward generalization of known results, 
states that Theorem \ref{thm:lcmodel_previous} also holds if $\overline{\M}_{0,n}$ is replaced by $\overline{\M}_{0,n\cdot(1/(k -\ell))}$. 
\begin{thm}
\label{thm:lcmodel}
Let $\a\in\Q$ with $2/(n-1) < \a \leq 1$ and $D_{k-\ell}\defeq\overline{\M}_{0,n\cdot (1/(k-\ell))}\setminus \M_{0,n}$ for $\ell\geq 0$.
Then, the log canonical model
\[\overline{\M}_{0,n\cdot(1/(k-\ell))}(\a) \defeq \mathrm{Proj} \bigoplus_{s\geq 0} H^0(\overline{\M}_{0,n\cdot(1/(k-\ell))}, s(K_{\overline{\M}_{0,n\cdot(1/(k-\ell))}} + \a D_{k-\ell} )) \]
   satisfies the following 
\begin{enumerate}
    \item If $2/(k+2)< \a \leq 2/(k+1)$ for $1\leq k \leq \lfloor (n-1)/2\rfloor$, then $\overline{\M}_{0,n\cdot(1/(k-\ell))}(\a)\cong \overline{\M}_{0, n\cdot (1/k)}$.
    \item If $2/(n-1) < \a \leq 2/(\lfloor n/2 \rfloor+1)$, then $\overline{\M}_{0,n\cdot(1/(k-\ell))}(\a)\cong (\P^1)^n//_{\OO(1,\cdots, 1)}\SL_2(\C)$.
\end{enumerate}
\end{thm}

For the readers' sake, we first give a brief outline of our approach.
The proof follows the same general strategy as in Section \ref{section: proof of the main theorem}, this time applied to the reduction morphism
\[
\rho : \overline{\M}_{0,n\cdot(1/(k-\ell))}\longrightarrow \overline{\M}_{0,n\cdot(1/k)}.
\]
We first compare the canonical divisor and the boundary divisors on the two Hassett spaces and rewrite
\[
K_{\overline{\M}_{0,n\cdot(1/(k-\ell))}}+\alpha D_{k-\ell}
\]
in terms of the pullback of a divisor on $\overline{\M}_{0,n\cdot(1/k)}$ together with an exceptional contribution.
A direct calculation shows that, in the relevant range of $\alpha$, this exceptional part is effective, while the divisor downstairs is precisely the one whose log canonical model is already known from the symmetric Hassett-Keel picture. Thus the argument reduces to comparing divisor classes under $\rho$ and applying the same method as in Section \ref{section: proof of the main theorem} to identify the resulting log canonical model.

Since $\overline{\M}_{0,n} \cong \overline{\M}_{0, n\cdot (1/2)}$ holds, Theorem \ref{thm:lcmodel} indeed generalizes Theorem \ref{thm:lcmodel_previous}.
       Note that $\overline{\M}_{0,n\cdot(1/(k-\ell))}$ is smooth and the map $\overline{\M}_{0,n\cdot(1/(k-\ell))}\to \overline{\M}_{0,n\cdot (1/k)}$ is (a composition of) smooth blow-ups along smooth transversal centres and $D_{k-\ell}$ is a simple normal crossing divisor; 
       see also \cite[Section 1]{KM11}.
The proof of Theorem \ref{thm:lcmodel} is achieved by comparing the canonical divisors of $\overline{\M}_{0,n\cdot(1/(k-\ell))}$ and $\overline{\M}_{0,n\cdot(1/k)}$.
\begin{proof}[Proof of Theorem \ref{thm:lcmodel}]

Let $m\defeq \lfloor n/2 \rfloor$, $\A\defeq n\cdot (1/k)$,  $\mathcal{B}\defeq n\cdot (1/(k-\ell))$, and $\rho:\overline{\M}_{0,\mathcal{B}} \to \overline{\M}_{0,\mathcal{A}}$ be the reduction morphism as in \cite[Definition 5.7]{KM11}.
By \cite[Lemma 5.3, Proposition 5.4]{KM11}, we have
\begin{align*}
    K_{\overline{\M}_{0,\A}} &= -\frac{2}{n-1}
D^{\A}(2) + \sum_{j \ge k+1}^m \left(-\frac{2}{n-1}\binom{j}{2} + j-2\right) D^{\A}(j),\\
\rho^*D^{\A}(2) &= D^{\mathcal{B}}(2) + \sum_{j=k-\ell+1}^{k} \binom{j}{2} D^{\mathcal{B}}(j), \\
\rho^*D^{\A}(j) &= D^{\mathcal{B}}(j)\quad (k+1 \le j \le m). \\
\end{align*}
Note that $\epsilon_{m-k}$ in \cite[Section 1]{KM11} 
 corresponds to $n\cdot (1/k)$ in our notation.
Let $D_k\defeq \overline{\M}_{0,\A}\setminus \M_{0,5}$.
Then
\begin{align*}
    A(\a, k-\ell, k)\defeq &\rho^*(K_{\overline{\M}_{0,\A}} + \a D_k)
    \\
    =&\left(\a-\frac{2}{n-1}\right)\sum_{j=2\ \mathrm{or}\ j\geq k-\ell+1}^{k}\binom{j}{2}D^{\mathcal{B}}(j) + \sum_{j\geq k+1}^m \left(\a-\frac{2}{n-1}\binom{j}{2}+j-2\right) D^{\mathcal{B}}(j).
\end{align*}
It follows that
\begin{align*}
    &K_{\overline{\M}_{0,\mathcal{B}}}+\a D_{k-\ell}-A(\a,k-\ell,k)\\
    &=\left(\alpha - \frac{2}{n-1}\right)D^{\mathcal{B}}(2) + \sum_{j\ge k-\ell+1}^m \left(\alpha - \frac{2}{n-1}\binom{j}{2} + j-2\right)D^{\mathcal{B}}(j)\\
    &-\left(\alpha-\frac{2}{n-1}\right)\sum_{\substack{j=2\\ j\ge k-\ell+1}}^k\binom{j}{2}D^{\mathcal{B}}(j) - \sum_{j\ge k+1}^m\left(\alpha-\frac{2}{n-1}\binom{j}{2} + j -2\right)D^{\mathcal{B}}(j)\\
    &= \sum_{j\ge k-\ell + 1}^k\left(\alpha + \frac{2}{n-1}\binom{j}{2} + j - 2\right)D^{\mathcal{B}}(j).
\end{align*}
Since $D^{\mathcal{B}}(j)$  is contracted by $\rho$ for $k-\ell+1\le j \le k$,
this implies that the divisor $K_{\overline{\M}_{0,\mathcal{B}}}+\a D_{k-\ell}-A(\a,k-\ell,k)$ is $\rho$-exceptional and effective.
Hence, by the same discussion as \cite[Section 2]{Sim08}, if $K_{\overline{\M}_{0,\mathcal{A}}} + \a D_{k}$ is ample, then $\overline{\M}_{0,\mathcal{B}}(\a)\cong \overline{\M}_{0,\mathcal{A}}$ ; see also \cite[Corollary 2.3.5]{Sim08}.
Since the ampleness of this divisor has already been shown by, for example \cite[Proposition 5.6]{KM11}, the result follows.
\end{proof}

\end{document}